\UseRawInputEncoding
\documentclass[12pt,draft,reqno]{amsart}

\usepackage[all]{xy}
\usepackage{amsthm,array,amssymb,amscd,amsfonts,latexsym}

\usepackage{enumerate}
\usepackage{mathrsfs}

\theoremstyle{plain}

\newtheorem{theorem}{Theorem}

\newtheorem{lemma}[theorem]{Lemma}
\newtheorem{remark}[theorem]{\bf Remark}

\newtheorem*{cornonumber}{Corollary}

\theoremstyle{definition}

\newtheorem{nothing*}[theorem]{}
\newtheorem{subnothing*}[sub]{}

\newtheorem*{exnonumber}{Example}

\theoremstyle{remark}

\newcommand{\cc}{\raise .4pt \hbox{{$\scriptstyle{\bullet}$}}}

\begin{document}

\title[
Underlying varieties, and group structures]{Underlying varieties,\\ and group structures}
\author[Vladimir L. Popov]{Vladimir L. Popov}
\address{Steklov Mathematical Institute,
Russian Academy of Sciences, Gub\-kina 8,
Moscow 119991, Russia}
\address{National Research University
``Higher School of Economics'', Mos\-cow,
Russia\newline
\ }
\email{popovvl@mi-ras.ru}

\dedicatory{To
the memory of
A.\;Bia{\l}ynicki-Birula}

\begin{abstract}


We explore to what extent the underlying variety of a connected algebraic group or the underlying manifold of a real Lie group determines its group structure.
 \end{abstract}

\maketitle

\noindent{\bf 1.\;Introduction.}
The central theme of this paper is the question as to what extent the underlying variety of a connected algebraic group or the underlying manifold of a real Lie group determines its group structure. The author was naturally led to consideration of it by the positive answer he gave in \cite{P2} to the question of B. Kunyavsky \cite{K} about the validity of the statement formulated below as Corollary of Theorem\;\ref{thm1}.\;This statement concerns the possibility to represent
the underlying variety of a connected reductive algebraic group in the form of a product of underlying varieties of its derived subgroup and connected component of the center. It follows from it, in particular, that the underlying variety of any connected reductive non-semisimple algebraic group can be represented as a product of algebraic varieties of positive dimension.

Sections 2, 3 make up the content of \cite{P2}, where
the possibility of such represen\-tations is explored. For some of them, in Theorem
\ref{thm1} is proved their existence, and in Theorems \ref {thm2}--\ref{thm6},
on the contrary, their non-existence.

Theorem \ref{thm1} shows that there are non-isomorphic reductive groups whose underlying varieties are isomorphic. In Sections 4--10, we explore the prob\-lem, naturally arising in connection with this, of dependence of the group structure of a connected algebraic group on the geometric properties of its underlying variety.\;A striking illustration of this dependence is the classical theorem about the commutativity of a connected algebraic group whose underlying variety is complete.\;In an exp\-li\-cit or implicit form, this problem was considered in the classical papers of A.\;Weil
\cite{W}, C.\;Chevalley \cite{Ch}, A.\;Borel \cite{Bor}, A.\;Grothendieck \cite[p.\,5-02, Cor.]{G}, M.\;Ro\-sen\-licht \cite[Thm.\,3]{R}, M.\;Lazar \cite[Thm.]{L}.

In Theorems
   \ref{toroidd}--\ref{torcr}, it is proved that such group characteristics of a con\-nected algebraic group as
  dimensions of its radical and unipotent ra\-di\-cal, reductivity, semisimplicity, solvability, unipotency, toricity, the property of being a semi-abelian variety
  can be expressed in terms of the
  geometric properties and numerical invariants of its
  underlying variety.

Theorem \ref{lem8} generalizes to the case of connected solvable affine algeb\-raic groups M. Lazar's theorem, which states that an algebraic group,
   whose underlying variety is isomorphic to an affine space, is unipotent.

    Theorem \ref{thm1}, when applied to connected semisimple algebraic groups (in contrast to its application to reductive non-semisimple ones), does not give a way to construct non-isomorphic such groups with isomorphic underlying varieties.
   However, in fact, such groups do exist: in Sections 6, 7 we find a method for constructing them.

It is well known (see \cite[\S4, Exer.\;18, p.\,122]{Bou}) that for $n\geqslant 7$, ​​there exist infinite (parametric) families of pairwise non-isomorphic $n$-dimen\-sional connected unipotent algebraic groups.\;Being isomorphic to the $n$-dimensional affine space $\mathbb A^n$, their underlying varieties are isomorphic to each other. We show that the situation is different for connected reductive algebraic groups: in Theorem \ref{thm8}, is proven that for any such group $R$,
 the number of all algebraic groups, considered up to isomorphism, whose underlying variety is isomorphic to the underlying variety of $R$, is finite.
 Generally speaking, this number is greater than $1$.\;We prove (Theorems \ref{thm8},
  \ref{thm9}) that if the group $R$ is either simply connected and semisimple, or simple, then it is equal to $1$.

The proof of  Theorem
   \ref{thm8} relies on the general
  finiteness theorem for
  the number of connected reductive algebraic groups of a fixed rank, considered up to isomorphism (Theorem \ref{finired});
  it is proved in the appendix (Section 10).\;Theorem \ref{finired}
is a fundamental fact of the theory of algebraic groups that is well known for semisimple groups (in which case it follows from the finiteness of their centers).\;However, in full ge\-ne\-ra\-lity (that is, for reductive, and not just semisimple groups), the author 
  failed to find it in the literature.

The obtained results imply similar results for connected compact real Lie groups, in particular, the finiteness theorem for the number of such groups of any fixed dimension (Theorem \ref{cogr}).\;As in the case of  Theorem \ref{finired},
   the author failed to find this fundamental fact of the theory of compact real Lie groups in the literature.\;The same applies to the finiteness theorem for the number of reduced root data of any fixed rank, which follows from the \ref{finired} theorem (Theorem \ref{rd}).

   The results of this paper were announced in \cite{P3}, \cite{P4}, \cite{P5}, and \cite{P6}.

   \vskip 2mm

{\it Acknowledgements.}   The author is grateful to T. Bandman, V. Gor\-batse\-vich, and Yu.\;Zarhin for
drawing his attention to the related pub\-li\-ca\-tions, and M. Brion, J.-P. Serre, and the referee for the com\-ments.

\vskip 3mm

{\it Conventions and notation.}

 We follow the point of view on algebraic groups adopted in \cite{Bo2}, \cite{Hu}, \cite{Sp}, and use the following notation:

\begin {enumerate} [\hskip 4.2mm $\cc$]
\item $k$ is an algebraically closed field, over which all algebraic varieties considered below are defined.
\item Groups are considered in multiplicative notation. The unit element of a group
$G$
is denoted by
$e$ (which group is meant will be clear from the context).
    \item For groups $G$ and $ H $, the notation $G\simeq H$ means that they
    are iso\-mor\-phic.
 \item ${\mathscr C}(G)$ is the center of a group $G$.
\item ${\mathscr D}(G)$ is the derived group of a group $G$.
    \item $\langle g \rangle$ is the cyclic group generated by an element $g$.
            \item A torus means affine algebraic torus, and a homomorphism of algebraic groups means  their algebraic homomorphism.
            \item ${\mathscr R}(G)$ and ${\mathscr R}_u (G)$ are, respectively, the radical and the unipotent radical of a connected  affine algebraic group $G$.
                 \item $G^\circ$ is the identity connected component
of an algebraic group or a Lie group $G$.
                \item ${\rm Lie}(G)$ is the Lie algebra of an algebraic group or a Lie group $G$.
                \item If $\alpha\colon G\to H$ is a homomorphism of algebraic groups, then
                \begin{equation}\label{differe}
                d_{{\boldmath \centerdot}}\hskip -.2mm\alpha\colon {\rm Lie}(G)\to {\rm Lie}(H)
                \end{equation}
                is its differential at the unit element.
                \item ${\mathbb G}_a$ и ${\mathbb G}_m$ are respectively, the one-dimensional additive and multipli\-cative algebraic groups.
                \item ${\rm Hom}(G, H)$ and $ {\rm Aut}(G)$ are the groups of algebraic homomor\-phisms if
$G$ and $H$ are algebraic groups.\;The character of such a group $G$ is
an element of the group ${\rm Hom}(G, \mathbb G_m)$.
  \item $\mathbb A_n $ is the $n$-dimensional coordinate affine space.
    \item $\mathbb A_*^n$ is the product of $n$ copies of
    the variety $\mathbb A^1 \setminus \{0\}$.
        \item Let $p: = {\rm char}(k)$ and $a \in \mathbb Z$. If $ap \neq 0$, then $a'$ denotes
the quotient of dividing $a$ by the greatest power of $p$ that divides $a$. If $ap=0$,
            then $a': = a$.
            \end{enumerate}

\noindent{\bf  2.\;Reductive groups with isomorphic group varieties.}
In this section, we prove the existence of some representations of underlying
va\-rie\-ties of affine algebraic groups in the form of products of algebraic varieties, and also the existence of connected non-isomorphic reductive non-semisimple algebraic groups whose underlying varieties are isomor\-phic algebraic varieties.

  Let  $G$ be a connected reductive algebraic group. Then
   \begin{equation*}
   D:={\mathscr D}(G)\quad\mbox{and}\quad Z:={\mathscr C}(G)^\circ
   \end{equation*}
  are respectively a connected semisimple algebraic group
and a torus (see \cite[Sect. 14.2, Prop.\,(2)]{Bo2}).
The algebraic groups $D \times Z$ and $G$ are not always isomorphic; the latter is equivalent to the equality $D \cap Z = \{e\}$, that, in turn,
is equivalent to the property that the isogeny of algebraic groups $D \times Z \to G, \;
(d, z) \mapsto d z $, is their isomorphism.

\begin{theorem}\label{thm1}
There is an injective homomorphism of algebraic groups
$\iota \colon Z \hookrightarrow G $ such
that the mapping
$$ \varphi \colon D \times Z \to G, \; (d, z) \mapsto d \! \cdot \! \iota (z), $$
is an iso\-mor\-phism of
algebraic varieties \textup(but, in general, not a homo\-mor\-phism
of algebraic groups\textup).
\end{theorem}
\begin{cornonumber}
\label{iv}
\textup The underlying varieties of \textup(in general, non-isomorphic\textup) algeb\-raic groups $D \times Z$ and $G$ are isomorphic algebraic varieties.
\end{cornonumber}
\begin{remark}{\rm
The existence of $\iota $ in the proof of Theorem \ref{thm1} is established by an explicit construction.}
\end{remark}

\begin{exnonumber}
[{\rm \cite[Thm.\;8,\;Proof]{P1}}]\label{exa}  Let $G={\rm GL}_n$. Then $D={\rm SL}_n$, $Z=\{{\rm diag}(t,\ldots, t)\mid t\in k^\times\}$, and one can take
$${\rm diag}(t,\ldots, t)\mapsto {\rm diag}(t, 1,\ldots, 1)$$ as $\iota$. In this example, if  $n\geqslant 2$, then $G$ and $D\times Z$ are non-isomorphic algebraic groups, because the center of $G$ is connected and that of $D\times Z$ is not.
\end{exnonumber}

\begin{proof}[Proof of Theorem {\rm\ref{thm1}}]
Let $T_D$ be a maximal torus of the group $D$, and let $T_G$ be a maximal torus of the
group $G$ containing $T_D$. The torus $T_D$ is a direct factor of the torus $T_G$, i.e., in the latter, there is a torus $S$ such that the map
$T_D \times S \to T_G, \; (t, s) \mapsto ts$, is an isomorphism of algebraic groups (see \cite [8.5, Cor.] {Bo2}).
We shall show that the mapping
\begin {equation} \label{psi}
\psi \colon D \times S \to G, \quad (d, s) \mapsto d s,
\end {equation}
is an isomorphism of algebraic varieties.

We have
(see \cite[Sect. 14.2, Prop. (1),(3)]{Bo2}):
\begin{equation}\label{pro}
\mbox{\rm (a)} \,\;Z\subseteq T_G,\quad
\mbox{\rm (b)}\,\; DZ=G,\quad \mbox{\rm (c)}\,\;\mbox{$|D\cap Z|<\infty$}.
\end{equation}

Let $g\in G$. In view of
\eqref{pro}{\rm(b)}, there are the elements $d\in D$, $z\in Z$ such that
$g=dz$,
and in view of
\eqref{pro}(a)
and the definiton of $S$, there are the elements
$t\in T_D$, $s\in S$ such that $z=ts$. We have $dt\in D$ and
$\psi(dt, s)=dts=g$. Therefore, the morphism $\psi$ is surjective.

Consider in $G$ a pair of mutually opposite Borel subgroups containing $T_G$. Their
unipotent radicals $U$ and $U^-$
lie in the group $D$. Let $N_D(T_D)$ and $N_G(T_G)$
be the normalizers of the tori $T_D$ and $T_G$ in the groups $D$ and $G$, respectively. Then
 $N_D(T_D)\subseteq N_G(T_G)$ in view of  \eqref{pro}{\rm(b)}. The homomor\-phism   $N_D/T_D\to N_G/T_G$  induced by this embedding is an isomorphism of groups  (see \cite[IV.13]{Bo2}), by which we identify them and denote by  $W$.
For every element $\sigma\in W$, we fix a representative $n_\sigma\in N_D(T_D)$.
The group $U\cap n_{\sigma}U^-n_{\sigma}^{-1}$ does not depend on the choice of this representative because $T_D$ normalizes the group $U^-$; we denote it by $U'_{\sigma}$.

It follows from the Bruhat decomposition that for every element $g\in G$,
there are uniquely defined  elements $\sigma\in W$, $u\in U$, $u'\in U'_\sigma$,  and
$t^{\ }_G\in T_G$ such that
$g=u'n_\sigma u t^{\ }_G$ (see \cite[28.4, Thm.]{Hu}).  In view of the definition of the torus $S$, there are uniquely defined  elements $t^{\ }_D\in T_D$ and $s\in S$ such that $t^{\ }_G=t^{\ }_Ds$, and since
 $u', n_\sigma, u, t^{\ }_D\in  D$, the condition $g\in D$ is equivalent to the condition $s=e$. It follows from this and the definition of the morphism $\psi$ that the latter is injective.

Thus,  $\psi$ is a bijective morphism. Therefore, to prove that it is an iso\-mor\-phism of algebraic varieties, it remains to prove its separability  (see \cite[Sect. 18.2, Thm.]{Bo2}). We have
${\rm Lie}\,(G)={\rm Lie}\,D+{\rm Lie}\,(T_G)$ (see \cite[Sect. 13.18, Thm.]{Bo2}) and  ${\rm Lie}\,(T_G)={\rm Lie}\, (T_D)+{\rm Lie}\,(S)$ (in view of the definition of the torus $S$). Therefore,
\begin{equation}\label{Lie}
{\rm Lie}\,(G)={\rm Lie}\,(D)+{\rm Lie}\,(S).
\end{equation}
On the other hand, from \eqref{psi}
it is obvious that the restrictions of the morphism $\psi$ to the subgroups $D\times\{e\}$ and $\{e\}\times S$
in $D\times S$, are isomor\-phisms respectively with
the subgroups $D$ and $S$ in $G$. Since
${\rm Lie}\,(D\times S)=
{\rm Lie}\,(D\times \{e\})
+ {\rm Lie}\,(\{e\} \times S),
$
from \eqref{Lie} it follows that
the differential of
the morphism  $\psi$ at the point $(e, e)$ is surjective.
Therefore (see   \cite[Sect. 17.3, Thm.]{Bo2}), the morphism $\psi$ is separable.

Since $\psi$ is an isomorphism, from \eqref{psi} it follows that $\dim (G)=\dim (D)\break +\dim (S)$. On the other hand, from
\eqref{pro}{\rm(b)},{\rm(c)}
it follows that $\dim (G)=\dim (D)+\dim (Z)$. Therefore,
$Z$ and $S$ are equidimensional and hence isomorphic tori. Consequently, as $\iota$ one can take the composition of any tori isomorphism $Z\to S$ with the identity embedding $S \hookrightarrow G$.
\end{proof}

\noindent{\bf  3.\;Properties of factors.}
In contrast to the previous section, this one, on the contrary, concerns the non-existence of some representations the underlying variety of an affine algebraic group as a product of algebraic varieties.

 \begin{theorem}\label{thm2} An algebraic variety, on which there is a non-constant
invert\-ible regular function, cannot be a direct factor of the underlying variety of a connected
semisimple algebraic group.
\end{theorem}

\begin{proof}
If the statement of Theorem \ref{thm2} were not true, then the existence
of the non-constant invertible regular function specified in it would imply the existence of such a function on a connected semisimple algebraic group. Dividing this function by its value at the unit element, we would then get, according to \cite[Thm.\;3]{R},
a non-trivial character of this group, which contradicts the absence
of non-trivial characters of any connected semisimple groups.
\end{proof}

Below, unless otherwise stated, we assume that $k = \mathbb C$.
By the Lefschetz principle, Theorems \ref{thm5}, \ref{thm6}, \ref{thm7}, \ref{thm8}, \ref{thm9}
proved below are valid for any field $k$ of characteristics zero.
Topological terms refer to classical topology, and homology and cohomology are singular.

Every complex reductive algebraic group $G$ has a compact real form, every two such forms are conjugate, and if ${\sf G}$ is one of them, then the topological manifold
$G$ is homeomorphic to the product of ${\sf G}$ and a Euclidean space; see
\cite[Chap.\,5, \S2, Thms.\,2, 8, 9]{OV}. Therefore,
$G$ and ${\sf G}$ have the same homology and cohomology.
This is used below without further explanation.

\begin{theorem}\label{thm3}
If a $d$-dimensional algebraic variety $X$
is a direct factor
of the underlying variety of a connected reductive algebraic group, then
\begin{equation*}
H_d(X, \mathbb Z)\simeq\mathbb Z\quad \mbox{and}\quad H_i(X, \mathbb Z)=0\;\; \mbox{for $i>d$}.
\end{equation*}
\end{theorem}

\begin{proof}
Suppose there is a connected reductive
algebraic group $G$ and an algebraic variety $Y$ such that the underlying variety of
the group $G$ is isomorphic to $X \times Y$. Let $n: = \dim(G)$; then $\dim(Y) = n-d$.
The algebraic varieties
$X$ and $Y$ are irreducible non-singular and affine. Therefore
(see \cite[Thm. 7.1]{M}),
\begin{equation}\label{H}
H_i(X, \mathbb Z)=0 \;\;\mbox{for $i>d$}, \;\;\;H_j(Y, \mathbb Z)=0\;\;\mbox{for $j>n-d$}.
\end{equation}
By the universal coefficient theorem, for every algebraic variety $V$ and every $i$, we have
\begin{equation}\label{HHH}
H_i(V, \mathbb Q)\simeq H_i(V, \mathbb Z)\otimes \mathbb Q,
\end{equation}
and by the K\"unneth formula,
\begin{equation}\label{HHHH}
H_n(G,\mathbb Q)\simeq H_n(X\times Y, \mathbb Q)\simeq \textstyle\bigoplus_{i+j=n}H_i(X, \mathbb Q)\otimes H_j(Y, \mathbb Q).
\end{equation}
Therefore, from \eqref{H}, \eqref{HHHH} it follows that
\begin{equation}\label{H-H}
H_n(G, \mathbb Q)\simeq H_d(X, \mathbb Q)\otimes H_{n-d}(Y, \mathbb Q).
\end{equation}

Consider a compact real form ${\sf G}$ of the group $G$. Since ${\sf G}$
is a closed connected orientable $n$-dimensional topological manifold,
$H_n({\sf G}, \mathbb Q) \simeq \mathbb Q$. Hence, $H_n(G, \mathbb Q) \simeq \mathbb Q $. From this and \eqref {H-H} it follows that $H_d (X, \mathbb Q)\break \simeq \mathbb Q$.
In turn, in view of
\eqref{HHH}, the latter implies $H_d (X, \mathbb Z) \simeq \mathbb Z $, because $H_d (X, \mathbb Z)$  is a finitely generated (see \cite [Sect. 1.3]{D}) torsion-free abelian group (see \cite [Thm. 1]{AF}).
\end{proof}

\begin{cornonumber}
A contractible algebraic variety \textup(in particular, $\mathbb A^d$\textup) of posi\-tive
dimension cannot be a direct factor of the underlying variety of
a con\-nected reductive algebraic group.
\end{cornonumber}

\begin{theorem}\label{thm5}
An algebraic curve cannot be a direct factor of the under\-lying variety of a connected
semisimple algebraic group.
\end{theorem}

\begin{proof}
Suppose an algebraic curve $X$ is a direct factor of the underlying
 variety of a connected semisimple algebraic group $G$. Then $X$ is irre\-ducible, non-singular, affine, and there is a surjective morphism $\pi \colon G \to X$. Due to rationality
of the underlying variety of $G$
 (see \cite[Sect. 14.14]{Bo2}), the existence of $\pi$ implies unirationality, and hence, by Luroth's theo\-rem, rationality of $X$. Therefore, $X$ is isomorphic to an open subset $U$ of $\mathbb A^1$.
 The case $U = \mathbb A^1$ is impossible due to Theorem \ref{thm3}. If $U \neq \mathbb A^1$, then there is a non-constant invertible regular function on $X$, which is impossible in view of  Theo\-rem\;\ref{thm2}.
\end{proof}

\begin{theorem}\label{thm6}
An algebraic surface cannot be a direct factor of the underlying variety of a
connected semisimple algebraic group.
\end{theorem}

\begin{proof}
Suppose there are a connected semisimple
algebraic gro\-up $G$ and the algebraic varieties $X$ and $Y$ (necessarily irreducible and smooth) such that
$X$ is a  surface, and the product $X \times Y$ is isomorphic to the
underlying variety of the group $G$. We keep the notation of the proof of Theorem \ref{thm3}. Since the group $G$ is semisimple, the group ${\sf G}$ is semisimple too.\;Hence, $ H^1({\sf G}, \mathbb Q) = H^2({\sf G}, \mathbb Q) = 0$ (see
\cite[\S9, Thm. 4, Cor. 1]{O}).
Insofar as
the $\mathbb Q$-vector spaces $H^i({\sf G}, \mathbb Q)$ and $H_i({\sf G}, \mathbb Q)$ are dual to each other, this gives
\begin{equation}\label{zero}
H_1(G, \mathbb Q)=H_2(G, \mathbb Q)=0.
\end{equation}
Since the group $G$ is connected, the topological manifolds $X$ and $Y$ are connected too. Therefore,
\begin{equation}\label{0}
H_0(X,\mathbb Q)=H_0(Y, \mathbb Q)=\mathbb Q.
\end{equation}
From \eqref{HHHH}, \eqref{zero}, and \eqref{0} it follows that
$H_2(X,\mathbb Q)=0$. In view of
\eqref{HHH}, this contradics Theorem \ref{thm3}, which completes the proof.
 \end{proof}

\begin{remark} {\rm Theorem \ref{thm5} can be proved in the same way as
Theo\-rem\;\ref{thm6}.\;Namely, in the proof of the latter one only needs to consider $X$ being a curve,
and then the arguments used in it lead to the equality $H_1(X,\mathbb Q)=0$, which contradicts Theorem \ref{thm3}. The other proof is given in the hope that
it raises the chances of carrying over Theorem \ref{thm5} to positive characteristic.}
\end{remark}

\noindent {\bf  4.\;Group properties determined by properties of underlying va\-riety.}
Theorem \ref{thm1}
naturally leads to the question as to what extent the underlying variety of an algebraic group determines its group struc\-ture.

Explicitly or implicitly, this question has long been considered in the literature.

For example, M. Lazar proved in \cite{L} that if the underlying variety of an algebraic group is isomorphic to an affine space, then this group is unipotent (for a short proof, see Remark \ref{Laz} below).

By Chevalley's theorem, every connected algebraic group $G$ contains the largest connected affine normal subgroup $G_{\rm aff} $, and the group $G/G_ {\rm aff}$ is an abelian variety.\;M.\;Rosenlicht in \cite{R} considered $G$
such
that $G_{\rm aff}$ is a torus; this property is equivalent to the absence
of connected one-dimensional unipotent subgroups in $G$.
In modern terminology (see \cite[Sect.\;5.4]{Br}), such groups are called
{\it semi-abelian varieties} (M. Rosenlicht called them {\it toroidal}).\;Next Theorem \ref{toroidd}
gives a criterion that the group $G$ is a
semi-abelian variety in terms of
geometric properties of its underlying variety
(the proof does not use the restriction on the characteristic of the field $k$; constraint (b) in Theorem \ref{toroidd} is weaker than the constraint made in \cite[Prop.\;5.4.5]{Br}):

\begin{theorem}[{{\rm semi-abelianness criterion}}]\label{toroidd}
The following properties of a connected algebraic group $G$ are equivalent:
\begin{enumerate}[\hskip 7.2mm \rm (a)]
\item[$\rm (sa_1)$] $G$ is toroidal;
\item[$\rm (sa_2)$] $G$ does not contain subvarieties isomorphic to ${\mathbb A}^1$.
\end{enumerate}
\end{theorem}

\begin{proof}
Let $\pi \colon G \to G / G_ {\rm aff}$ be the natural epimorphism,
and let $X$ be a subvariety of $G$ isomorphic to $\mathbb A^1$. Shifting it by an appropriate element of the group $G$, we can assume that the unit element $e$ of the group $G$ lies in $X$. Since the variety $X$ is isomorphic to the underlying variety of the group $\mathbb G_a$, we can endow the variety $X$ with a structure of an algebraic group isomorphic to group $\mathbb G_a$ with the unit element $e$.
Then $\pi|_X\colon X\to G/G_{\rm aff}$ is
a homomorphism of algebraic groups in view of
\cite[Thm.\,3]{R}. Since $X$ is an affine and $G/G_{\rm aff}$ is a complete algebraic variety, this yields $X\subseteq G_{\rm aff}$. Therefore, the matter comes down to proving
the equivalence of the following properties:
\begin{enumerate}[\hskip 6.2mm\rm(a')]
\item[$\rm (sa_1')$] $G_{\rm aff}$ is a torus;
\item[$\rm (sa_2')$] $G_{\rm aff}$ does not contain subvarieties isomorphic to ${\mathbb A}^1$.
\end{enumerate}

$\rm (sa_1')
\Rightarrow
\rm (sa_2')$: Let the subvariety $X$ of the torus $G_ {\rm aff}$ be isomorphic to ${\mathbb A}^1$. The algebra of regular functions on $G_ {\rm aff} $ is generated by invertible functions.\;This means that this is also the case for the algebra of regular functions on $X$.\;This contradicts the fact that there are no non-constant invertible regular functions on
$\mathbb A^1$.

$\rm (sa_2')
\Rightarrow
\rm (sa_1')$: If $\rm (sa_2')$ holds, then
$G_ {\rm aff}$ is reductive, since the variety
$\mathscr R_u (G_{\rm aff})$ is isomorphic to $\mathbb A^d$ (see \cite[p.\,5-02, Cor.]{G}), which for $d>0$ contains affine lines. In addition, ${\mathscr D}(G_{\rm aff})\!=\!\{e\}$, because
root subgroups in a semisimple group are isomorphic to ${\mathbb G}_a$, whose underlying variety is isomorphic to $\mathbb A^1$. Hence,
$G_{\rm aff}$ is a torus.
\end{proof}

\begin{cornonumber}
The following properties of
a connected algebraic group
 $G$ are equivalent:
\begin{enumerate}[\hskip 3.0mm\rm(a)]
\item in the underlying variety of the group $G$, there are no subvarieties isomorphic to ${\mathbb A}^1$;
\item in the group $G$, there are no algebraic subgroups isomorphic to\;${\mathbb G}_a$.
\end{enumerate}
\end{cornonumber}

\begin{proof}
According to \cite[Prop.]{R}, property (b) is equivalent to $G$ being
semi-abelian variety.
Therefore, the claim follows from Theorem \ref{toroidd}.
\end{proof}

Below is listed a series of group properties of connected affine algeb\-raic groups determined by the properties of their underlying varieties. In the formulations of the corresponding statements,
the following numerical invariants of underlying varieties are used.

Let $X$ be an irreducible algebraic variety.\;The multiplicative group $k[X]^\times$ of invertible
regular functions on $X$ contains the subgroup of non-zero constants $k^\times$, and the quotient
$k[X]^\times/k^\times$
is a  free abelian group of a finite rank (see \cite[Thm. 1]{R}).
Let us denote
 \begin{equation}\label{units}
{\rm units}(X):={\rm rank}(k[X]^\times /k^\times).
 \end{equation}
According to \cite[Thms. 2, 3]{R}, this invariant
has the following proper\-ties:

\begin{enumerate}[\hskip 2.0mm\rm(i)]
\item
If $X$ and $Y$ are irreducible algebraic varieties, then
\begin{equation}\label{times}
{\rm units}(X\times Y)={\rm units}(X)+{\rm units}(Y).
\end{equation}
\item
If $G$ is a connected algebraic group, then
\begin{equation}\label{Hom}
{\rm units}(G)={\rm rank}\big({\rm Hom}(G, \mathbb G_m)\big).
\end{equation}
\end{enumerate}

\begin{lemma}\label{inequ} Let $G$ be a connected algebraic group.
Then
\begin{enumerate}[\hskip 4.2mm\rm(i)]
\item ${\rm units}(G)\leqslant \dim(G)$;
\item the equality in {\rm(i)}
is equivalent to the property that
$G$ is a torus.
\end{enumerate}
\end{lemma}

\begin{proof} By \cite[Cor.\;5 of Thm.\;16]{Ros}, the kernel of every character of the group $G$
contains the smallest normal algebraic subgroup  $D$ of $G$ such that the group $G/D$ is affine. In view of this and \eqref{Hom}, in what follows we can (and will) assume that $G$ is affine.\;Similarly, since ${\mathscr R}_u(G)$ lies in the kernel of every character of $G$, we can (and will) assume that $G$ is reductive.\;Let $T$ be a maximal torus of  $G$.\;Since the set $\bigcup_{g\in G}gTg^{-1}$ is dense in  $G$ (see \cite[Sect.\;12.1, Thm.(a),(b) and Sect. 13.17, Cor.\,2(c)]{Bo2}), the restriciton of characters of $G$ to $T$ is a group embedding ${\rm Hom}(G,
\mathbb G_m)\hookrightarrow {\rm Hom}(T,
\mathbb G_m)$; whence, in view of \eqref{Hom} and \cite[Sect.\,8.5, Prop.]{Bo2}, we get ${\rm units}(G)\leqslant
{\rm units}(T)=\dim(T)\leqslant \dim(G)$. This completes the proof.
\end{proof}

In what follows, we use the following notation:
\begin{equation}\label{defmu}
{\rm mh}(X):=\max\{d\in \mathbb Z_{\geqslant 0}\mid H_d(X, \mathbb Q)\neq 0\}.
\end{equation}
If $X$ is a non-singular affine algebraic variety, then, according to
 \cite[Thm. 7.1]{M},
$${\rm mh}(X)\leqslant \dim(X).$$

\begin{theorem}\label{ddi} If $G$ is a connected affine algebraic group, then
\begin{align}\label{ddm}
\dim ({\mathscr R}_u(G))&=\dim(G)-{\rm mh}(G),\\
\label{rddm}
\dim ({\mathscr R}(G))&=\dim(G)-{\rm mh}(G)+{\rm units}(G).
\end{align}
\end{theorem}

\begin{proof} By \cite[p.\,5-02, Cor.]{G}, the underlying variety of the group ${\mathscr R}_u(G) $ is isomorphic to an affine space.\;Therefore, the underlying varieties of the groups
 $G$ and $R: = G/{\mathscr R}_u (G)$, considered as topological manifolds, are homotopy equivalent. Therefore, $H_i (G, \mathbb Q) \simeq H_i (R, \mathbb Q)$ for every $i$, and hence
\begin{equation}\label{muGL}
{\rm mh}(G)={\rm mh}(R).
\end{equation}
Since the group $R$ is reductive, it follows from \eqref{HHH} and Theorem \ref{thm3} that \begin{equation}\label{mdL}
{\rm mh}(R)=\dim(R).
\end{equation}
In view of $\dim(R)=\dim(G)-\dim ({\mathscr R}_u(G))$,  equalities
\eqref{muGL} and \eqref{mdL} imply \eqref{ddm}.

The group ${\mathscr R}(G)$ is a semi-direct product of its maximal torus  $T$ and the group ${\mathscr R}_u(G)$ (see\,\cite[Sect.\,10.6, Thm.]{Bo2}), so
\begin{equation}\label{sspr}
\dim ({\mathscr R}(G))=\dim(T)+\dim({\mathscr R}_u(G)).
\end{equation}
Let $\pi\colon G\to G/{\mathscr R}_u(G)$ be the canonical projection. Then
(see \cite[Sect.\,11.21]{Bo2})
\begin{equation}\label{centto}
\pi(T)=\pi({\mathscr R}(G))={\mathscr C}(G/{\mathscr R}_u(G))^{\circ}.
\end{equation}
Since the group $G/{\mathscr R}_u(G)$ is reductive, it follows from \eqref{Hom} and \eqref{centto}
that
\begin{equation}\label{ravv}
\begin{split}
{\rm units}(G)&={\rm rank}\big({\rm Hom}(G, \mathbb G_m)\big)\\
&=
{\rm rank}\big({\rm Hom}(G/{\mathscr R}_u(G), \mathbb G_m)\big)\\
&=
\dim\big({\mathscr C}(G/{\mathscr R}_u(G))^{\circ}\big)\\
&=
\dim\big(\pi(T)\big)=\dim(T).
\end{split}
\end{equation}
Now equality \eqref{rddm} follows from \eqref{ddm}, \eqref{sspr}, and \eqref{ravv}.
\end{proof}

Since reductivity (respectively, semisimplicity) of a connected affine algebraic group is equivalent to the triviality of its unipotent radical (respectively, radical), Theorem \ref{ddi} gives the following criteria for re\-duc\-ti\-vity and semisimplicity in terms of the geometric properties of the underlying variety:

\begin{theorem}[{{\rm reductivity criterion}}]\label{crre}
The following properties of a con\-nec\-ted affine algebraic group $G$ are equivalent:
\begin{enumerate}[\hskip 9.2mm \rm(a)]
\item[$\rm(red_1)$] $G$ is reductive;
\item[$\rm(red_2)$] $\dim(G)={\rm mh}(G)$.
\end{enumerate}

If these properties hold, then
$\dim(\mathscr C(G))={\rm units}(G)$.
\end{theorem}

\begin{proof} The first claim follows from \eqref{ddm}, and the second from \eqref{centto} and \eqref{ravv}.
\end{proof}

\begin{theorem}[{{\rm semisimplicity criterion}}]
The following properties of a con\-nec\-ted affine algebraic group $G$ are equivalent:
\begin{enumerate}[\hskip 7.2mm \rm(a)]
\item[$\rm(ss_1)$] $G$ is semisimple;
\item[$\rm(ss_2)$]
$\dim(G)={\rm mh}(G)-{\rm units}(G)$;
\item[$\rm(ss_3)$] $\dim(G)={\rm mh}(G)$ and ${\rm units}(G)=0$.
\end{enumerate}
\end{theorem}

\begin{proof}
$\rm(ss_1)\Leftrightarrow\rm(ss_2)$
follows from \eqref{rddm}.
In view of reductivity of
  semi\-simple groups and finiteness of their centers,
  from Theorem \ref{crre} it follows that
$\rm(ss_1)\Rightarrow\rm(ss_3)$.
Clearly,
$\rm(ss_3)\Rightarrow\rm(ss_2)$.
\end{proof}

The following Theorem \ref{lem8} generalizes M. Lazar's theorem \cite{L} to the case of connected solvable affine algebraic groups and shows that solvability of a connected affine algebraic group also can be charac\-te\-rized in terms of the geometric properties of its
underlying variety.

\begin{theorem}[{{\rm solvability criterion}}]\label{lem8}
The following properties of a con\-nec\-ted affine algebraic group $S$ are equivalent:
\begin{enumerate}[\hskip 9.2mm\rm (a)]
\item[$\rm(sol_1)$] $S$ is solvable;
\item[$\rm(sol_2)$] ${\rm mh}(S)={\rm units}(S)$;
\item[$\rm(sol_3)$] there are nonnegative integers
 $t$ and $r$ such that the underlying variety of the group
$S$ is isomorphic to
${\mathbb A}_*^t\times {\mathbb A}^r$, and in this case, neces\-sarily $t={\rm units}(S)$.
\end{enumerate}

If these properties hold, then the dimension of maximal tori of the group
  $S$ is equal to
${\rm units}(S)$, and the following equality holds
\begin{equation}\label{formu}
\dim({\mathscr R}_u(S))=\dim(S)-{\rm units}(S).
\end{equation}
\end{theorem}

\begin{proof}
$\rm(sol_1)\Leftrightarrow\rm(sol_2)$:
Let $G:=S/{\mathscr R}_u(S)$; it is a connected reductive algebraic group.\;We shall use the same notation as in the proof of Theorem \ref{thm1}.\;Solvability of the group $S$  is equivalent to the equality $G=Z$, whence, in view of connectedness of the groups $G$ and $Z$, it follows that
\begin{equation}\label{c-ri}
\mbox{$S$ is solvable} \iff \dim (G)=\dim (Z).
\end{equation}
Given
\eqref{ddm} and \eqref{muGL},
we have
\begin{equation}\label{muG}
\dim(G)={\rm mh}(S).
\end{equation}

The elements of the group
 ${\rm Hom}(S, \mathbb G_m)$ (respectively, ${\rm Hom}(G, \mathbb G_m)$) are trivial on the group ${\mathscr R}_u(S)$ (respectively, $D$). It follows follows this and
  \eqref{pro}(b)
that
\begin{equation}\label{1-5}
\begin{split}
{\rm Hom}(S, \mathbb G_m)
&
\simeq
{\rm Hom}(G, \mathbb G_m),
\\
{\rm Hom}(G, \mathbb G_m)
&
\simeq {\rm Hom}(Z/(Z\cap D), \mathbb G_m).
\end{split}
\end{equation}
From \eqref{Hom}, \eqref{1-5}, and \eqref{pro}(c)
we obtain
\begin{equation}\label{18}
\begin{split}
{\rm units}(S)
&={\rm rank}\big({\rm Hom}(Z/(Z\cap D), \mathbb G_m)\big)\\
&=\dim(Z/(Z\cap D))=\dim (Z).
\end{split}
\end{equation}
Matching \eqref{c-ri}, \eqref{muG}, and \eqref{18} completes the proof
the equivalence
$\rm(sol_1)\Leftrightarrow\rm(sol_2)$.

$\rm(sol_1)\Rightarrow\rm(sol_3)$:
This is proven in \cite[p.\,5-02, Cor.]{G} for the field $k$ of any characteristic.

$\rm(sol_3)\Rightarrow\rm(sol_2)$:
Let $\rm(sol_3)$ holds.
It follows from \eqref{times}, \eqref{Hom}, and the evident equality
${\rm units}({\mathbb A}^r)=0$ that
\begin{equation}\label{t}
{\rm units}({\mathbb A}_*^t\times {\mathbb A}^r)=t.
\end{equation}
On the other hand, since the topological manifold $\mathbb A^r$ is contractible,
and ${\mathbb A}_*^t$ is homotopically equivalent to the product  $t$ circles, we have
\begin{equation*}\label{KttK}
H_j({\mathbb A}_*^t\times {\mathbb A}^r, \mathbb Q)=\begin{cases}
\mathbb Q, &\mbox{при $j=t$},\\
0 & \mbox{при $j>t$};
\end{cases}
\end{equation*}
we conclude from this and  \eqref{defmu} that
\begin{equation}\label{tt}
{\rm mh}({\mathbb A}_*^t\times {\mathbb A}^r)=t.
\end{equation}
Comparing \eqref{t} with \eqref{tt} completes the proof of
the implication \break
$\rm(sol_3)\Rightarrow\rm(sol_2)$, and with it the proof of the first claim of the theorem.

Let the properties specified in the first statement of the theorem be satisfied.
Then it follows from the property
$\rm(sol_1)$ and \cite[Sect. 10.6, Thm.\,(4)]{Bo2} that the dimension of
maximal tori in $S$ is equal to
$\dim(S/{\mathscr R}_u(S))=\dim (G)$, what is equal to  ${\rm units}(S)$
in view of \eqref{c-ri} and  \eqref{18}. Equality \eqref{formu}
follows from equalities  \eqref{ddm} and
$\rm(sol_2)$. This proves the second claim of the theorem.

The group $S$ is unipotent  (respectively, is a torus) if and only if
it is solvable  (i.e., by (c), its underlying variety is isomorphic to
 ${\mathbb A}_*^t\times {\mathbb A}^r$), and by Theorem  \ref{ddi}, the equality
  ${\rm mh}(S)=0$
(respectively,
${\rm mh}(S)=\dim(S)$) holds. Now the last claim of the theorem follows from
\eqref{tt}.
\end{proof}

\begin{remark}\label{Laz}{\rm
Here is a short proof of M. Lazard's theorem
\cite{L}, suitable for the field $k$ of any characteristic.}

\begin{proof}[Proof of M. Lazard's theorem {\rm \cite{L}}]
Let the underlying variety of the group $G$ be iso\-mor\-phic to $\mathbb A^r$.  If $G$ is not unipotent,
then $G$ contains a non-identity semisimple element, and therefore, also a non-identity torus (see \cite[Thms.\,4.4(1), 11.10]{Bo2}). The action of this torus on  $G$ by left translations has no fixed points. This contradicts the fact that every algebraic torus action on $\mathbb A^r$  has a fixed point, see \cite[Thm.\,1]{Bi}.
\end{proof}
\end{remark}

The next two theorems show that unipotency and toricity of a con\-nected affine algebraic group can also be characterized in terms of the introduced numerical invariants of its underlying variety.

\begin{theorem}[{{\rm unipotency criterion}}]\label{unicr}
The following properties of a connected affine algebraic group
$G$ are equivalent:
\begin{enumerate}[\hskip 6.2mm\rm (a)]
\item[$\rm(u_1)$] $G$ is unipotent;
\item[$\rm(u_2)$] ${\rm mh}(G)={\rm units}(G)=0$;
\item[$\rm(u_3)$] the underlying variety of the group
$G$ is isomorphic to
${\mathbb A}^{\dim(U)}$.
\end{enumerate}
\end{theorem}

\begin{proof}
In view of solvability of unipotent group, the equivalence\linebreak  $\rm(u_1)\Leftrightarrow\rm(u_2)$
(respectively,
$\rm(u_1)\Leftrightarrow\rm(u_3)$)
follows from the equivalence
$\rm(sol_1)\Leftrightarrow\rm(sol_2)$
(respectively,
$\rm(sol_1)\Leftrightarrow\rm(sol_3)$)
and formula \eqref{formu} in Theorem \ref{lem8}.
\end{proof}

In the following theorem, the characteristic of the field $k$ can be arbitrary.

\begin{theorem}[toricity criterion]\label{torcr}
The following properties of a con\-nected affine algebraic group
 $G$ are equivalent:
\begin{enumerate}[\hskip 6.2mm\rm (a)]
\item[$\rm(t_1)$] $G$ is a torus;
\item[$\rm(t_2)$]  $\dim(G)={\rm units}(G)$;
\item[$\rm(t_3)$]  the underlying variety of the group
$G$ is isomorphic to
${\mathbb A}_*^{\dim(G)}$.
\end{enumerate}
\end{theorem}

\begin{proof}
Lemma \ref{inequ} gives $\rm(t_1)\Leftrightarrow\rm(t_2)$. The implication
$\rm(t_3)\Rightarrow\rm(t_2)$ follows from
\eqref{times} and ${\rm units}(\mathbb A_*^1)=1$,
and $\rm(t_1)\Rightarrow\rm(t_3)$ is evident.
\end{proof}

\noindent {\bf 5.\;Different group structures on the same variety.}\label{5}
As is known (see \cite[\S4, Exer.\,18, p.\,122]{Bou}),
there are infinitely many pairwise non-isomorphic connected unipotent algebraic groups of any fixed
dimension $\geqslant 7$; their underlying varieties,
however, all are isomorphic to each other (see Theorem \ref{lem8}). On the other hand, there are types of con\-nected algebraic groups for which
underlying varieties define
group structure unambiguously. The following Theorem
\ref{toroid} shows that such types in\-clude
semi-abelian varieties (the proof is not
uses restrictions on the characteristic of the field $k$). Below we will find other types of algebraic groups that have the indicated uniqueness property, see Theorems
\ref{thm8}(b) and\;\ref{thm9}.

\begin{theorem}\label{toroid} Let $G_1$ and $G_2$ be algebraic groups, one of which is a semi-abelian variety. The following properties are equivalent:
\begin{enumerate}[\hskip 4.2mm\rm(a)]
\item
the underlying varieties of the groups $G_1$ and $G_2$ are isomorphic;
\item the algebraic groups $G_1$ are $G_2$ are isomorphic.
\end{enumerate}
\end{theorem}
\begin{proof} \hskip -1mm Let $G_1$ be a semi-abelian variety. Then by
 \cite[Thm.\,3]{R}, the com\-position of an isomorphism of the underlying varieties
$G_2\to G_1$ with a suitable left translation of
$G_1$ is an isomorphism of algebraic groups, which proves
 (a)$\Rightarrow$(b).
\end{proof}

\begin{cornonumber}\label{trtr}
Isomorphness of algebraic groups, among which there is either a torus or an abelian variety, is equivalent to isomorphness of their underlying varieties.
\end{cornonumber}

From this, in particular, one obtains the fact, discovered by
A. Weil, that isomorphness of abelian varieties is equivalent to isomorphness of their underlying varieties (see \cite{W}).

\begin{remark}{\rm Semi-abelian varieties are commutative.\;In preprint \cite{Di} published after preprint \cite{P3} of the first version of this paper, it is proved that, for ${\rm char}(k)=0$, isomorphness of the underlying varieties of two connected commutative algebraic groups implies isomorphness of these algebraic groups. The existence of Witt groups shows that in this statement the condition ${\rm char}(k)=0$  cannot be dropped.}
\end{remark}

We now investigate the problem of determinability of group struc\-tu\-re by the properties of underlying variety for reductive algeb\-raic groups.

\begin{theorem}\label{thm7} Let $G_1$ and $G_2$ be connected affine algebraic groups, and let $R_i$ be a maximal reductive algebraic subgroup of $G_i$, $i=1, 2$.
If the underlying varieties of the groups $G_1$ and $G_2$ are isomorphic, then
the Lie algebras of the connected algebraic groups
$R_1$ and $R_2$ are isomorphic.
\end{theorem}

\begin{proof} From ${\rm char}(k)=0$ it follows that the group $G_i$ is a semidirect product of the groups $R_i$ and  ${\mathscr R}_u(G_i)$ (see \cite[11.22]{Bo2}).
Hence the group $R_i$ is connected (because $G_i$ is connected), and
the underlying manifolds of the groups
 $G_i$ and $R_i$ are homotopy equivalent
 topological manifolds
(see the proof of Theorem \ref{ddi}).

Consider a compact form ${\sf R}_i$ of the reductive algebraic group $R_i$.
The underlying manifolds of the groups
$R_i$ and ${\sf R}_i$ are homotopy equivalent.

Suppose that the underlying varieties of the groups $G_1$ and $G_2$ are isomor\-phic algebraic varieties, and therefore, homeomorphic topolo\-gi\-cal manifolds.\;Then the underlying manifolds of the groups ${\sf R}_1$ and
${\sf R}_2 $ are homotopy equivalent topological manifolds.\;In view of  \cite[Satz]{Sc}, this implies that the real Lie algebras ${\rm Lie} \, ({\sf R}_1)$ and ${\rm Lie} \, ({\sf R}_2) $ are isomor\-phic. Now the claim of the theorem follows from the fact that the real Lie algebra ${\rm Lie} \, ({\sf R}_i)$ is a real form of the complex Lie algebra ${\rm Lie} \, (R_i)$.
\end{proof}

\begin{theorem}\label{thm8}
Let $R$ be a connected reductive algebraic group.

   {\rm (i)} If $G$ is an algebraic group such that the underlying varieties of
$G$ and $R$ are isomorphic, then
\begin{enumerate}[\hskip 9.2mm\rm(a)]
\item $G$ is connected and reductive, and the Lie algebras
${\rm Lie}\,(R)$ and ${\rm Lie}\,(G)$ are isomorphic;
\item in the case of a semisimple simply connected group
 $R$, the algebraic groups $R$ and $G$ are isomorphic.
\end{enumerate}

{\rm (ii)} The number of all algebraic groups, considered up to isomor\-phism, whose underlying varieties are isomorphic to that of
 $R$, is finite.
\end{theorem}

\begin{proof} (i)(a) It follows from connectedness of the group $R$ and the con\-di\-tion on the group $G$
that the group $G$ is connected. In view of  Theorem \ref{thm7} and reductivity of the group $R$, the Lie algebra of a maximal reductive subgroup in $G$ is isomorphic to ${\rm Lie}\,(R)$. In parti\-cu\-lar, the dimension of this subgroup is equal to $\dim (R)$. Since $\dim(R) = \dim (G) $, this subgroup coincides with $G$.

(i)(b) From the condition on the group $G$ and simply connectedness of the underlying
manifold of the group $R$ it follows that the underlying
manifold of the group $G$ is simply con\-nected. In view of {\rm (a)}, the Lie algebras ${\rm Lie}\,(R)$ and ${\rm Lie}\,(G)$ are isomorphic. Consequently, the algeb\-raic groups $R$ and $G$ are iso\-mor\-phic (see \cite[Chap.\,1, \S3, $3^\circ$, Chap.\,3, \S3,
$4^\circ$]{OV}).

Statement (ii) follows from (i)(a) and finiteness of
the numbers of all, considered up to isomorphism,
connected reductive algebraic groups of a fixed dimension
This finiteness theorem, which the author failed to find in the li\-te\-ra\-ture, is proved below in appendix (Section \ref{Ap}, Theorem \ref{finired}, and Remark \ref{r5}).
\end{proof}

\begin{remark} {\rm Using the proof of Theorem \ref{finired} given in Section \ref{Ap}, one can obtain an upper bound for the number specified in Theorem \ref{thm8}(ii)
 (see also Remark\,\ref{up} below).
 }
\end{remark}

Theorem \ref{thm1} provides examples of
non-isomorphic connected reductive non-semisimple algebraic groups whose underlying varieties are
isomor\-phic  (according to Theorem\;\ref{thm8}(a), the Lie algebras of these groups are iso\-morphic).\;In the case of connected semisimple algebraic groups (that is, when $Z=\{e\}$), Theorem \ref{thm1} degenerates into
a trivial statement that does not provide such examples.\;However, non-isomorphic connected semisimple algebraic groups whose group varieties are  isomorphic do exist.\;Below is described a method that allows one to construct them.
It is suitable for a field $k$ of any characteristic.

\vskip 2mm

\noindent{\bf 6.\;Construction of non-isomorphic semisimple algebraic gro\-ups with
isomorphic underlying varieties.}
Fix a positive integer
$n$ and an abstract group $H$. Consider the group
\begin{equation*}
G:=H^{n}:=H\times\cdots\times H\quad \mbox{($n$ factors).}
\end{equation*}
We have
${\mathscr C}(G)=
{\mathscr C}(H)^{n}$.

Let $F_n$ be a free group of rank  $n$ with a free system of generators
 $x_1,\ldots, x_n$. For any elements $g=(h_1, \ldots, h_n)\in G$, where $h_j\in H$,  and  $w
\in F_n$, denote by $w(g)$ the element of $H$, which is the image of the element $w$
under the homomorphism $F_n\to H$, mapping $x_j$ to $h_j$ for every\;$j$.

Every element $\sigma \in {\rm End}(F_n)$
determines the map
\begin{equation} \label{mor}
{\widehat{\sigma}}\colon G\to G,\quad g\mapsto (\sigma(x_1)(g),\ldots, \sigma(x_n)(g)).
\end {equation}
It is not hard to see that
\begin{equation}\label{actio}
\begin{split}
\widehat{\sigma\circ \tau}
&=\widehat{\tau}\circ \widehat{\sigma}\;\,\mbox{для любых $\sigma, \tau\in {\rm End}(F_n)$},\\
\widehat{{\rm id}_{F_n}}&={\rm id}_{G}.
\end{split}
\end{equation}
It follows from \eqref{mor}
and the definition of $w(g)$ that

\begin{enumerate}[\hskip 2.2mm\rm(i)]
\item
$\widehat{\sigma}(S^{n})
\subseteq {S^{n}}$ для любой подгруппы $S$ в $H$;
\item $\widehat{\sigma}(gz)=\widehat{\sigma}(g)\widehat{\sigma}(z)$\;\;для любых $g\in G, z\in
{\mathscr C}(G)$.
\end{enumerate}
In particular, the restriction of the map $\widehat{\sigma}$
to the group ${\mathscr C}(G)$
is its endomorphism.

From \eqref{actio} it follows that if $\sigma \in {\rm Aut}(F_n) $, then
$\widehat{\sigma}$ is a bijection (but, in general, not an automorphism of the group $G$).\;Moreover, if $H$ is an algebraic group (respectively, a Lie group), then
$\widehat{\sigma}$ is an automorphism of the algebraic variety (respectively, a diffeomorphism of the differentiable manifold) that is $G$.

Consider now an element $\sigma \in {\rm Aut}(F_n)$ and
a subgroup of $C$ in ${\mathscr C}(G)$. Then from (ii) it follows
$C$-equivariance of the bijection $\widehat{\sigma}\colon G\to G$ if we assume that
every element $c\in C$ acts on the left
copy of $G$ as the translation (multiplication) by $c$, and on the right
one as the translation  by $\widehat{\sigma}(c)$.\;The quotient for the first action is the group
$G/C$, and for the second is the group $G/\widehat{\sigma}(C)$. Hence
 $\widehat{\sigma}$ induces a bijection $G/C\to G/\widehat{\sigma}(C)$.\;Moreover, if $H$  is an algebraic group (respectively, real Lie group), then this bijection is
 an isomorphism of algebraic varieties (respectively, a diffeomorphism of differentiable manifolds); see \cite[6.1]{Bo2}. Thus, $G/C$ and $G/\widehat{\sigma}(C)$  are isomorphic algebraic varieties (respec\-tively, diffeomorphic differentiable manifolds). But,
in general, $G/C$ and $G/\widehat{\sigma}(C)$ are not isomorphic as algebraic groups (respectively, as Lie groups).

Indeed, take for $H$ a simply connected semisimple algebraic group (re\-spec\-tively, a real compact real Lie group).\;Then $G$ is also a simply connected algeb\-raic group (respectively, a real compact Lie group), so the group ${\mathscr C}(G)$ is finite.
 Consider the natural epimomorphisms
$\pi\colon G\to G/C$ and $\pi_{\widehat{\sigma}(C)}\colon G\to G/\widehat{\sigma}(C)$.  Since the group $C$ is finite, the differentials
\begin{equation*}\label{isomo}
d_{{\boldmath \centerdot}}\pi_C\colon {\rm Lie}(G)\to {\rm Lie}(G/C)\quad\mbox{и}\quad
d_{{\boldmath \centerdot}}\pi_{\widehat{\sigma}(C)}\colon {\rm Lie}(G)\to {\rm Lie}(G/\widehat{\sigma}(C))
\end{equation*}
are the Lie algebra isomorphisms.\;Suppose there is an isomorphism of algebraic groups  (respectively, real Lie groups) $\alpha\colon G/C\to G/\widehat{\sigma}(C)$. Then
\begin{equation*}\label{cis}
(d_{{\boldmath \centerdot}}\pi_{\widehat{\sigma}(C)})^{-1}\circ d_{{\boldmath \centerdot}}\alpha\circ d_{{\boldmath \centerdot}}\pi_C\colon {\rm Lie}(G)\to {\rm Lie}(G)
\end{equation*}
is the Lie algebra automorphism of ${\rm Lie}(G)$. Since $G$ is simply connected, it has the form
$d_{{\boldmath \centerdot}}\widetilde{\alpha}$ for some automorphism
$\widetilde{\alpha}\in {\rm Aut}(G)$
(see \cite[Thm.\,6, p.\,30]{OV}). It follows from the construction that the diagram
\begin{equation*}
\begin{matrix}
\xymatrix@C=6mm@R=8mm{
G\ar[r]^{\widetilde{\alpha}}\ar[d]_{\pi_C}& G\ar[d]^{\pi_{\widehat{\sigma}(C)}}\\
G/C\ar[r]^{\hskip -4.5mm \alpha}& G/\widehat{\sigma}(C)
}
\end{matrix}
\end{equation*}
is commutative, which, in turn, implies that  $\widetilde{\alpha}(C)=\widehat{\sigma}(C)$.

Thus, the algebraic groups (respectively, real Lie groups) $G/C$ and $G/\widehat{\sigma}(C)$ are isomorphic
if and only if the subgroups $C$ and $\widehat{\sigma}(C)$ of the group $G$
lie in the same orbit of the
natural action of the group
${\rm Aut}(G)$ on the set
of all subgroups
of the group
${\mathscr C}(G)$.\;This action is reduced to the action of the group ${\rm Out}(G)$
(isomor\-phic to the automorphism group of
the Dynkin diagram of the group  $G$; see \cite[Chap.\,4, \S4, no.\,1]{OV}) because the group ${\rm Int}(G)$ acts on
${\mathscr C}(G)$ trivially. It is not difficult to find
$H$, $\sigma$, and $C$ such that the groups $C$ and $\widehat{\sigma}(C)$ do not lie in
the same ${\rm Out}(G)$-orbit. Here is a concrete example.

\vskip 2mm

\noindent {\bf 7.\;Example.} Consider a simply connected simple algebraic (res\-pecti\-vely, real compact Lie) group $H$ with a nontrivial center.
Take $n=2$, so that
\begin{equation}\label{HH}
G=H\times H.
\end{equation}
Let the element $\sigma\in {\rm End}(F_2)$ be defined by the equalities
\begin{equation}\label{sig}
\sigma(x_1)=x_1,\quad \sigma(x_2)=x_1x_2^{-1};
\end{equation}
clearly, $x_1, x_1x_2^{-1}$ is a free system of generators of the group
 $F_2$, so $\sigma\in{\rm Aut}(F_2)$). Let $S$ be a non-trivial
subgroup of the group ${\mathscr C}(H)$. Take
\begin{equation}\label{Ce}
C:=\{(s, s)\mid s\in S\}.
\end{equation}
Then from \eqref{mor}, \eqref{sig}, \eqref{Ce} it follows that
\begin{equation}\label{sCe}
\widehat{\sigma}(C)=\{(s, e)\mid s\in S\}.
\end{equation}
In view of simplicity of the group $H$, the elements of ${\rm Out}(G)$ carry out
permutations of
the factors on the right-hand side of  the equality
\eqref{HH}.
This,
\eqref{Ce}, and \eqref{sCe} imply
that $C$
and $\widehat{\sigma}(C)$ do not lie in the same
${\rm Out}(G)$-orbit.
Therefore,
\begin{equation*}
G/C=(H\times H)/C\quad \mbox{and}\quad G/\widehat{\sigma}(C)=(H/S)\times H
\end{equation*}
are non-isomorphic connected semisimple algebraic groups (respecti\-vely, real compact Lie groups), whose underlying varieties (respectively, manifolds)  are isomorphic (respectively, diffeomorphic).

For example, let $H\!=\!{\rm SL}_d$,
$d\!\geqslant\! 2$, and $S\!=\!\langle z\rangle$, where
$z\!=\!{\rm diag}(\varepsilon,\ldots, \varepsilon) \break\in H$,  $\varepsilon\in k$ is a primitive $d$-th root of  $1$.\;In this case, we obtain non-isomorphic algebraic groups
\begin{equation*}
G/C=({\rm SL}_d\times {\rm SL}_d)/\langle(z,z)\rangle, \quad G/\widehat{\sigma}(C)={\rm PSL}_d\times {\rm SL}_d,
\end{equation*}
whose underlying varieties are isomorphic.\;Note that if $d=2$, then
$G={\rm Spin}_4$, $G/C={\rm SO}_4$.

For $H={\sf SU}_d$ and the same
group  $S$, we obtain that
\begin{equation*}
G/C=K_1:=({\sf SU}_d\times {\sf SU}_d)/C,\quad G/\widehat{\sigma}(C)=K_2:={\sf PU}_d\times {\sf SU}_d
\end{equation*}
are  non-isomorphic connected semisimple compact real Lie groups who\-se underlying mani\-folds are diffeomorphic. For $d=p^r$ with prime $p$, this is proved in \cite[p.\,331]{BB}, where non-isomorphness of the groups $K_1$ and $K_2$ is deduced from the non-isomorphness of their Pontryagin rings  $H_*(K_1, \mathbb Z/p\mathbb Z)$ and $H_*(K_2, \mathbb Z/p\mathbb Z)$ (discribing these rings is a non-trivial problem). Note that if $d=2$, then
\begin{equation}\label{rHut}
K_1={\sf SO}_4,\quad K_2={\sf SO}_3\times {\sf SU}_2.
\end{equation}
That the underlying manifolds
\eqref{rHut} are diffeomorphic was known long ago:
in \cite[Chap.\,3, \S3.D]{Ha}, a diffemorphism between them
is const\-ructed by means of
quaternions.

\vskip 2mm

\noindent{\bf 8.\;The case of connected simple algebraic groups.}
The following theorem
shows that the phenomenon under exploration is not possible
for simple groups.

\begin{theorem}\label{thm9}
Let
$G_1$ and $G_2$ be algebraic groups, one of which is connected and simple.
The following properties are equivalent:
\begin{enumerate}[\hskip 4.2mm\rm(a)]
\item the underlying varieties of the groups
$G_1$ and $G_2$ are isomorphic;
\item algebraic groups $G_1$ and $G_2$ are isomorphic.
\end{enumerate}
\end{theorem}

\begin{proof} Let  the group $G_1$ be connected and simple.

Suppose (a) holds. Let $\widetilde G_1$ be a simply connected
algebraic group with the Lie algebra isomorphic to
 ${\rm Lie}\,(G_1)$.  Then $G_1$ is isomorphic to ${\widetilde G}_1/Z_1$ for some subgroup $Z_1$ of ${\mathscr C}(\widetilde G_1)$.
From Theorem \ref{thm8} it follows that the group $G_2$ is isomorphic to ${\widetilde G}_1/Z_2$ for some subgroup $Z_2$ of
${\mathscr C}(\widetilde G_1)$. As explained above, statement (b) is equivalent to the property that the subgroups $Z_1$ and $Z_2$ lie in the same orbit
of the natural action of the group
${\rm Out}({\widetilde G}_1)$
(isomorphic to the automorphism group of the Dynkin diagram
of the group  $\widetilde G_1$)
on the set of all subgroups of the group
${\mathscr C}(\widetilde G_1)$.
We shall show that the subgroups
$Z_1$ and $Z_2$ indeed lie in the same ${\rm Out}(\widetilde G_1)$-orbit.

Since the fundamental groups of topological manifolds $G_1$ and $G_2$ are isomorphic to, respectively, $Z_1$ and $Z_2$, it follows from (a) that
that the finite groups $Z_1$ and $Z_2$ are isomorphic. Let $d$ be their order.

The structure of the group ${\mathscr C}(\widetilde G_1)$ is known (see \cite[Table\,3, pp. 297--298]{OV}). Namely, if the type of the simple group $\widetilde G_1$ is different from
\begin{equation} \label{Df}
{\sf D}_\ell \; \; \mbox {with even $\ell \geqslant 4$,}
\end{equation}
then ${\mathscr C}(\widetilde G_1)$ is a cyclic group. In the case of the group
$\widetilde G_1$ of type \eqref{Df}, the group
${\mathscr C}(\widetilde G_1)$ is isomorphic to the Klein four-group
$\mathbb Z /2\mathbb Z \oplus \mathbb Z /2\mathbb Z$.
Since there is at most one subgroup of a given finite order in
any cyclic group, we get that if the type of $\widetilde G_1$ is different from \eqref{Df},
then $Z_1 = Z_2$,
so in this case, the subgroups $Z_1$ and $Z_2$ lie in the same ${\rm Out}(G_1)$-orbit.

Now, let  $\widetilde G_1$ be of type \eqref{Df}.\;This means that $\widetilde G_1={\rm Spin}_{4m}$ for some integer $m\geqslant 2$.\;Since $|{\mathscr C}(\widetilde G_1)|=4$, only the cases $d=1, 2, 4$ are possible. It is clear that $Z_1=Z_2$ for $d=1$ and $4$,  so in these cases, as above, the subgroups
$Z_1$ and $Z_2$  lie in the same ${\rm Out}(G_1)$-orbit.
Therefore, it remains to consider only the case
$d=2$.

There are exactly three subgroups of order $2$ in ${\mathscr C}(\widetilde G_1)$. The natural action on ${\mathscr C}(\widetilde G_1)$ of the group ${\rm Out}({\rm Spin}_{4m})$ (isomorphic to the auto\-mor\-phism group
of the Dynkin diagram of the group $\widetilde G_1$)
can be easily described explicitly using the information specified
in \cite[Table 3, p. 297--298]{OV}\footnote{In the notation of  \cite[Table\,3, p. 297--298]{OV}, for $ m = 2 $, each permutation of the vectors $h_1, h_3, h_4$ with fixed $h_2$ is realized by some automorphism of the Dynkin dia\-gram (identified with the corresponding outer automorphism), and for $m>2$, the only non-trivial automorphism of the Dynkin diagram swaps $ h_ {2m} $ and $ h_ {2m-1} $ and leaves all the rest $h_i$'s fixed.}.\;This description shows that
the number of ${\rm Out}({\rm Spin}_ {4m})$-orbits
on the set of these subgroups equals $1$ for $ m = 2 $ and equals $2$ for $ m> 2 $. Thus,
for $m = 2$, the groups $G_1$ and $G_2$ are isomorphic and it remains for us to consider the case $m>2$.

The quotient group of the group ${\rm Spin}_ {4m}$ by a subgroup of order 2 in
${\mathscr C} (\widetilde G_1)$, which is not fixed (respectively, fixed) with respect to the group
${\rm Out}({\rm Spin}_{4m})$, is the half-spin group ${\rm SSpin}_ {4m}$ (respectively, the orthogonal group ${\rm SO}_{4m}$).
 Let $ {\sf SSpin}_ {4m}$ and
${\sf SO}_{4m}$ be the compact  real forms of the groups ${\rm SSpin}_{4m}$ and
${\rm SO}_{4m}$, respectively. If the underlying varieties of the groups ${\rm SSpin}_{4m}$ and ${\rm SO}_{4m}$ were isomorphic, then the underlying manifolds of the groups
 ${\sf SSpin}_{4m}$ and
${\sf SO}_{4m}$ were homo\-topy equi\-valent.\;But accord\-ing to
\cite[Thm.\,9.1]{BB},
 for $m>2$, they are not homo\-topy equivalent because $H^*({\sf SSpin}_ {4m}, \mathbb Z /2\mathbb Z)$ and $H^*({\sf SO}_ {4m}, \mathbb Z /2\mathbb Z)$ for $m> 2$ are not isomorphic as algebras over the Steenrod algebra\footnote {Note that
 $H^*({\sf SSpin}_ {n}, \mathbb Z /2\mathbb Z) $ and $H^*({\sf SO}_ {n},
\mathbb Z /2\mathbb Z) $ are isomorphic as algebras over
$\mathbb Z /2\mathbb Z$ if (and only if)
 $n$ is a power of $2$, see \cite [p.\, 330]{BB}.}.
Hence the underlying varieties of  the groups ${\rm SSpin}_ {4m}$ and $ {\rm SO}_{4m}$ for $m>2$ are not isomorphic.\;This completes the proof of implication (a)$\Rightarrow$(b).\;Impli\-cation (b)$\Rightarrow$(a) is obvious.
\end{proof}

Considerations used in the proof of Theorem \ref{thm9}
yield a proof of the following Theorem \ref{BB}, which was published in \cite{BB} without proof.

\begin{theorem}[{{\rm  \cite[Thm. 9.3]{BB}}}]\label{BB}
The underlying manifolds of two con\-nected real compact simple Lie groups are
homotopy equivalent
if and only if these Lie groups
are isomorphic.
\end{theorem}

\begin{proof}
It repeats the proof of Theorem \ref{thm9} if in it one
assumes that $G_1$ and $G_2$ are connected real compact simple
Lie groups, whose underlying manifolds are  homo\-topy equivalent, and replaces ${\rm Spin}_{4m} $,
${\rm SSpin}_ {4m}$, and $ {\rm SO}_{4m} $, respectively, with
${\sf Spin}_{4m}$, ${\sf SSpin}_{4m}$, and $ {\sf SO}_{4m} $.
\end{proof}

\noindent {\bf 9.\;Questions.}\

1. Previous considerations naturally
lead to the question of finding a classification of pairs of non-isomorphic connected reductive algebraic gro\-ups, whose underlying varieties are isomorphic.
Is it possible to obtain\;it?

2. The same for connected real compact Lie groups, whose underlying mani\-folds are homotopy equivalent.

3. It seems plausible that, using, in the spirit of  \cite{Bo1}, \'etale cohomo\-logy in place of singular homology and cohomology, it is possible  to prove Theorems \ref{thm5} and
\ref{thm6} and implication (c)$\Rightarrow$(a) of Theorem \ref{lem8} in the case of
positive characteristic
of the field $k$.
Are Theorems \ref{thm7}, \ref{thm8}, \ref{thm9} true for such $k$?

4. The author is not aware of examples of connected simple algebraic groups whose underlying variety is a product of algebraic varieties of positive dimension. Do they exist (N.\;L.\;Gor\-de\-ev's question)?

5. The problems considered in this paper are obviously reformulated taking into account rationality questions, i.e., definability of varieties over an algebraically non-closed field $\ell$.\;How are the results of this paper modified in this context? Some of them, for example, Theorem \ref{toroid}, do not change: the above proof of this theorem, with the added remark that the specified left shift is performed by a rational over $\ell$ element, is transferred to the context of definability over $\ell$ and shows that two algebraic groups defined over $\ell$, one of which is a semi-abelian variety, are isomorphic over $\ell$ if and only if their underlying varieties are isomorphic over $\ell$.\;In particular, for tori or abelian varieties defined over $\ell$, their isomorphness over $\ell$
 is equivalent to  isomorphness  of their underlying varieties over $\ell$.

6. The results of this paper
concern the uniqueness of the structure of
a connected reductive algebraic group (respectively, real compact Lie group) on an algebraic variety (respectively, differentiable compact manifold) that admits at least one such structure.\;Is it possible to give a criterion for the existence of at least one such structure on an algebraic variety (respectively, differentiable compact manifold)  in terms of its geometric characteristics?

\vskip 2mm

\noindent{\bf 10.\;Appendix: finiteness theorems for connected reductive al\-geb\-raic groups and compact Lie groups.}
 \label{Ap} In this section, the characteristic of  $k$ can be arbitrary.

\begin{theorem}\label{finired}
The number of all, considered up to isomorphism,
con\-nected reductive algebraic groups
of any fixed rank is finite.
\end{theorem}

\begin{proof}
For every connected reductive group $G$, there is a torus $Z$ and a simply con\-nected semisimple algebraic group $S$ such that the group $G$
is the quotient group of the group $Z \times S$ by a finite central subgroup. Indeed, let $S$
 be the universal covering group of the connected semi\-simple group
${\mathscr D}(G) $, let $\pi \colon S \to {\mathscr D} (G)$  be the natural projection, and
  let $ Z = {\mathscr C}(G)^\circ $.\;Then the map
  $Z \times S \to G$,
  $(z, s) \mapsto z \! \cdot \! \pi (s)$ is an epimorphism with a finite kernel, i.e., the natural projection to the quotient group by a finite central subgroup.

Being simply connected, the group $S$ is, up to isomorphism, uniquely determined by the type of its root system. Insofar as the set of types of root systems of any fixed rank is finite, tori of the same dimension are isomorphic, and $\mathscr C(S)$ is a finite group,
the problem comes down to proving that,
although for $\dim (Z)>0$  there are infinitely many
finite subgroups $F$ in $ \mathscr C(Z \times S)$, the set of all,
up to isomorphism,
groups of the form $ (Z \times S)/F$ is finite. Note that for every element
$\sigma \in {\rm Aut} (Z \times S) $, the groups $(Z\times S)/F$ and $(Z\times S)/\sigma (F)$
are isomorphic.

Proving this, we put
\begin{equation*}
n: = \dim (Z)> 0,
\end{equation*}
and let $\varepsilon_1,\ldots, \varepsilon_n$ be a basis of the group ${\rm Hom}(Z, \mathbb G_m)\simeq \mathbb Z^n$.

For every positive integer $r$, denote by ${\mathcal D}_{r \times n}$ the set of all matrices $(m_ {ij}) \in
{\rm Mat}_{r \times n} (\mathbb Z) $ such that
\begin{enumerate}[\hskip 4.2mm \rm (a)]
\item $ m_{ij} = 0$ for $ i \neq j $;
\item $ m_{ii} $ divides $m_{i + 1, i + 1} $;
\item $ m_{ii} = m_{ii} '$ (see the notation in Section 1).
\end{enumerate}

Consider a matrix $M= (m_{ij})\in {\rm Mat}_ {r \times n} (\mathbb Z) $. Then
\begin{equation} \label {matri}
Z_M: = \textstyle \bigcap_{i = 1}^{r}{\rm ker}(\varepsilon_1^{m_ {i1}}
\cdots \varepsilon_n^{m_ {in}})
\end {equation}
is an algebraic $\big(n - {\rm rk}(M) \big)$-dimensional
subgroup of the group $Z$.
Every algebraic subgroup of the group $Z$ is obtained in this way.
If  the matrix $M=(m_{ij})$ shares properties (a) and (b),
then $Z_M = Z_{M'}$, where $M': = (m_ {ij} ') $,
because ${\rm ker} (\varepsilon_i^{d})={\rm ker}(\varepsilon_i^{d'}) $.
If $ r = n $, and the matrix $ M $ is non-degenerate and shares properties (a), (b), (c), then $ Z_M $
is a finite abelian group with the invariant factors $|m_{11 }|, \ldots\break
\ldots, |m_ {nn}| $.

The elementary transformations of rows of the matrix $M$ do not chan\-ge
the group $Z_M$. If $\tau_1, \ldots, \tau_n$ is another basis of the group\linebreak
$ {\rm Hom} (Z, \mathbb G_m) $, then $\tau_i = \varepsilon_1^{c_ {i1}}
\cdots \varepsilon_n^{c_ {in}}$, where
$C=(c_ {ij}) \in {\rm GL}_r (\mathbb Z)$.
The automorphism of the group ${\rm Hom} (Z, \mathbb G_m) $,
which maps $ \varepsilon_i$ to $\tau_i $ for every $i$,
has the form $ \sigma_ {C}^{*} $, where $ \sigma_C $  is an automorphism of the group $Z$. The mapping $ {\rm GL}_n (\mathbb Z) \to {\rm Aut}(Z) $, $ C \mapsto \sigma_C $, is a group isomorphism and the following equality holds:
\begin{equation}\label{sC}
Z_{MC}=\sigma_C(Z_M)
\end{equation}
Since the elementary transformations of the columns of the matrix $M$ are realized
by multiplying the matrix $M$ on the right by the corresponding mat\-rices from
$ {\rm GL}_r (\mathbb Z) $, and by means of the elementary
transformations of rows and columns the matrix $M$ can be transformed into its Smith diagonal
normal form, \eqref{sC} implies the existence of an automorphism $ \nu
\in {\rm Aut}(Z) $ and a matrix $ D \in \mathcal D_ {r \times n} $ such that
$ \nu (Z_M) = Z_D$.

Now consider a finite subgroup $F$ in $ \mathscr C (Z \times S) =
Z \times {\mathscr C} (S)$ and the canonical projections
\begin{equation*}
Z \xleftarrow{\pi_Z} F
\xrightarrow{\pi_S} \mathscr C (S).
\end{equation*}
The groups $ (Z \times S)/F$ and
$ \big((Z \times S)/(F \cap Z) \big)/\big(F/(F\cap Z)\big)$
are isomorphic.\;Being an $n$-dimensional torus,
the group $Z /(F \cap Z) $ is isomorphic to the torus $ Z $. Therefore, the groups
$ (Z \times S)/(F \cap Z)$ and $ Z \times S $ are isomorphic.\;Hence, without changing, up to isomorphism,
the group $(Z \times S)/F$,
we can (and will) assume that $ F \cap Z = \{e\}$.\;Then $ {\rm ker} (\pi_S) = \{e\}$, and therefore,
$ \pi_S $ is an isomorphism between $F$ and the subgroup $ \pi_S(F) $ in the group $ \mathscr C(S) $.
Let $ \alpha \colon \pi_S (F) \to \pi_Z (F) $ be an epimorphism that is the composition of the isomorphism inverse to $ \pi_S $ with $ \pi_Z $. Then
\begin{equation*}
F=
\{\alpha(g)\cdot g\mid g\in \pi_S(F)\}.
\end{equation*}

The subgroup $ \pi_Z (F) = \alpha (\pi_S (F) )$ in $ Z $ is finite and therefore has the form
$ Z_M $ for some non-degenerate matrix $ M \in {\rm Mat}_{n \times n} (\mathbb Z) $.
According to the above,
there is an element $ \nu \in {\rm Aut} (Z) $ such that $ \nu (\pi_Z (F)) = Z_D $, where $ D $
is a  non-degenerate
matrix from
$ \mathcal D_{n \times n} $; we denote by the same letter the extension of $ \nu $ up to an element of $ {\rm Aut} (Z \times S) $, which is the identity on $S$. Replacing
the group $ F $ by the group $ \nu (F) $ shows that, without changing,
  up to isomorphism,
  the group $ (Z \times S) / F $,
we can (and will) assume that $ \pi_Z (F) = Z_D $.

Thus, if $ \mathscr F $ is the set
all subgroups
in $ Z \times \mathscr C (S) $ of the form
\begin{equation*}
\{\gamma (h) \cdot h \mid h \in H \},
\end{equation*}
where $H$ runs through all subgroups of $\mathscr C(S) $, and $ \gamma $ through all epimor\-phisms $ H \to Z_D $ with a non-degenerate matrix $ D \in \mathcal D_ {n \times n} $, then $ F \in \mathscr F $.
Since the group $ \mathscr C(S) $ is finite, and the order of the group $ Z_D $ is $ |\det (D)| $, the set $ \mathscr F $ is finite. This completes the proof.
\end{proof}

\begin{theorem}\label{rd}
The number of all, considered up
to isomorphism,
root data of any
fixed rank is finite.
\end{theorem}

\begin{proof}
This follows from Theorem \ref{finired} because
connected reduc\-ti\-ve al\-geb\-raic groups are classified by their root data, see
\cite[Thms.\,9.6.2, 10.1.1]{Sp}.
\end{proof}

\begin{theorem}\label{cogr}
The number of all, considered
up to isomorphism,
con\-nected real compact Lie groups
of any fixed rank is finite.
\end{theorem}

\begin{proof}
This follows from Theorem \ref{finired} in view of the correspondence between connected reductive algebraic groups and connected
real com\-pact Lie groups, given by passing to a real compact form, see \cite[Thm. 5.2.12]{OV}.
\end{proof}

\begin{remark}\label{up}{\rm
Using this proof of Theorem \ref{finired}, one can obtain an upper bound for
the numbers specified in it and in Theorems \ref{rd} and \ref{cogr}.}
\end{remark}

\begin{remark}\label{r5} {\rm Since the rank does not exceed the dimension of the group, Theorem \ref{finired}  (respectively, Theorem \ref{cogr}) shows that
the number of all,
considered up to an iso\-morphism\-m, connected
reductive algebraic groups (respectively, connected real compact Lie groups) of any fixed dimension is finite.}
\end{remark}

\end{document}